\documentclass[12pt, twoside]{amsart}
\usepackage[english]{babel}
\theoremstyle{plain}
\addtocounter{page}{1}
\usepackage{graphicx, color}
\usepackage{cancel}

\newtheorem{theorem}{Theorem}[section]

\newtheorem{lemma}{Lemma}[section]

\def \R {{\mathbb {R}}}

\title[Multiplicity of solutions for the Minkowski-curvature equation]{Multiplicity of solutions for the Minkowski-curvature equation\\ via shooting method}

\author[A. Boscaggin]{Alberto Boscaggin}
\address{Alberto Boscaggin\newline\indent 
Dipartimento di Matematica ``Giuseppe Peano''
\newline\indent
Universit\`a di Torino
\newline\indent
via Carlo Alberto 10, 10123 Torino, Italia}
\email{alberto.boscaggin@unito.it}

\author[F. Colasuonno]{Francesca Colasuonno}
\address{Francesca Colasuonno\newline\indent
Dipartimento di Matematica ``Giuseppe Peano''
\newline\indent
Universit\`a di Torino
\newline\indent
via Carlo Alberto 10, 10123 Torino, Italia}
\email{francesca.colasuonno@unito.it}

\author[B. Noris]{Benedetta Noris}
\address{Benedetta Noris
\newline \indent Laboratoire Ami\'enois de Math\'ematique Fondamentale et Appliqu\'ee\newline\indent
Universit\'e de Picardie Jules Verne\newline\indent
33 rue Saint-Leu, 80039 AMIENS, France}
\email{benedetta.noris@u-picardie.fr}

\thanks{ 
}

\begin{document}

\maketitle

%

\begin{abstract}
{In this paper we prove the existence and the multiplicity of radial positive oscillatory solutions for a nonlinear problem governed by the mean curvature operator in the Lorentz-Minkowski space. The problem is set in a ball $B_R$ of $\mathbb R^N$ and is subject to Neumann boundary conditions. The main tool used is the shooting method for ODEs.}

\medskip \noindent {\sc{Sunto.}} 
{In questo lavoro dimostriamo esistenza e molteplicit\`a di soluzioni oscillanti, radiali e positive che oscillano di un problema non-lineare governato dall'operatore di curvatura media nello spazio di Lorentz-Minkowski. Il problema \`e ambientato in una palla $B_R$ di $\mathbb R^N$ ed è soggetto a condizioni di Neumann al bordo. Il principale strumento usato \`e il metodo di {\it shooting} per le EDO.}


\end{abstract}

\maketitle

\section{Introduction}
We consider the following Neumann problem 
\begin{equation}\label{eq:main}
\left\{
\begin{array}{ll}
\vspace{0.1cm}
\displaystyle{-\mathop{\rm div}\left(\frac{\nabla u}{\sqrt{1-|\nabla u|^2}}\right) = f(u)} & \mbox{ in } B_R \\
\vspace{0.1cm}
u > 0,\quad u \mbox{ radial} & \mbox{ in } B_R \\
\partial_\nu u = 0 & \mbox{ on } \partial B_R, \\
\end{array}
\right.
\end{equation}
where $\nu$ is the outer unit normal of $\partial B_R$ and $B_R\subset\mathbb R^N$ ($N\geq1$) the ball of radius $R$ centered at the origin. Since we are interested in radial solutions, with the usual abuse of notation, we will often write $u(x) = u(r)$ for $r = \vert x \vert$.

The operator $-\mathop{\rm div}\left(\frac{\nabla u}{\sqrt{1-|\nabla u|^2}}\right)$ that governs the equation is usually referred to as mean curvature operator in the Lorentz-Minkowski space. It naturally arises in several problems of Differential Geometry and General Relativity \cite{BaSi-8283,Ecker,Ge-83}, and also in the Born-Infeld theory of Electromagnetism \cite{BI,BoCoFo-pp,BdAP-16,BoIa-pp}. 
In the last decades, the interest in problems involving the Minkowski-curvature operator has increased also in the field of Nonlinear Analysis. Existence and multiplicity results for this class of problems have been proved both in bounded and unbounded domains, both under Dirichlet boundary conditions and under Neumann boundary conditions (see, among others, \cite{Az-14,Az-16,BeJeMa-09,BeJeTo-13,BeJeTo-13a,BeMa-07,BoFe-pp,BoGa-ccm,CoCoObOm-12,CoelhoCorsatoRivetti,CoObOmRi-13,DW,Ma-13,ABF3} and the references therein). In particular, in \cite{ABF3}, under suitable assumptions on $f$, we proved the existence of pairs of oscillatory solutions of \eqref{eq:main}, via shooting method. The aim of the present paper is to cover a class of nonlinearities that was not allowed in \cite{ABF3}.

We assume that the nonlinearity $f$ satisfies the following hypotheses
\begin{itemize}
\item[$(f_\textrm{reg})$] $f \in \mathcal{C}^1([0,+\infty))$;
\item[$(f_\textrm{eq})$] $f(0) = f(1) = 0$, $f(s) < 0$ for $0 < s < 1$ and $f(s) > 0$ for $s > 1$;
\item[$(f_1)$] $f'(1)=0$.
\end{itemize}

We observe that assumption $(f_\textrm{eq})$ is motivated by the fact that, under Neumann boundary conditions, no positive solution to \eqref{eq:main} exists if $f$ has constant sign. 
Therefore, we assume that $f$ vanishes at some point ($s=1$ without loss of generality) and we note that, as a consequence, problem \eqref{eq:main} always admits the constant solution $u\equiv1$. We look for non-constant solutions.
 
An example of admissible nonlinearity $f$ satisfying $(f_\textrm{reg})$, $(f_\textrm{eq})$ and $(f_1)$ is the function $f:\,[0,\infty)\to\mathbb R$ defined as 
$$
f(s) = s(s-1)^3.
$$
Before stating the main result of this paper, we recall the multiplicity result proved in \cite{ABF3}.\smallskip

\noindent{\bf Theorem} (Theorem 1.1 of \cite{ABF3}) {\it Let $f$ satisfy $(f_\textrm{reg})$, $(f_\textrm{eq})$ and
\begin{itemize}
\item[$(f_1)'$] $f'(1) > \lambda_{k+1}^{\textnormal{rad}}$ for some integer $k \geq 1$,
\end{itemize}
where $\lambda_{k+1}^{\mathrm{rad}}$ is the $k$-th non-zero radial eigenvalue of the Laplacian in $B_R$ with Neumann boundary conditions. 
Then there exist at least $2k$ distinct non-constant radial solutions $u_1^\pm,\ldots,u_{k}^\pm$ to \eqref{eq:main}. 
Moreover, we have
\begin{itemize}
\item[(i)] $u_j^+(0)>1$ for every $j=1,\ldots,k$;
\item[(ii)] $u_j^-(0)<1$ for every $j=1,\ldots,k$;
\item[(ii)] $u_j^-(r)-1$ and $u_{k+1-j}^+(r)-1$ have exactly $j$ zeros for $r\in(0,R)$, for every $j=1,\ldots,k$.
\end{itemize}
}

\smallskip

In the present paper, we prove the following theorem. 

\begin{theorem}\label{thm:main}
Let $f$ satisfy $(f_\textrm{reg})$, $(f_\textrm{eq})$ and $(f_1)$. Then, for every integer $k \geq 1$ there exists a threshold radius $R_k^*>0$ such that, if $R\ge R_k^*$, problem \eqref{eq:main} admits $4k$ distinct non-constant solutions $u_1^{\pm},\ldots,u_{2k}^{\pm}$. 
Moreover, we have
\begin{itemize}
\item[(i)] $u_j^{+}(0)>1$ for every $j=1,\ldots,2k$;
\item[(ii)] $u_j^{-}(0)<1$ for every $j=1,\ldots,2k$;
\item[(ii)] $u_j^{\pm}(r)-1$ and $u_{2k+1-j}^{\pm}(r)-1$ have exactly $j$ zeros for $r\in(0,R)$, for every $j=1,\ldots,k$.
\end{itemize}
\end{theorem}
We compare now the two results. Firstly, we note that when $(f_1)'$ is in charge, $(f_1)$ is never satisfied; the prototype nonlinearity for Theorem 1.1 of \cite{ABF3} is $f(s)=s^{q-1}-s$, with $q>2+\lambda_{k+1}^{\textnormal{rad}}$. On the other hand, the two assumptions are clearly not complementary: the case $0<f'(1)\le \lambda_2^{\mathrm{rad}}$ is still left out. Actually, the reasoning for proving Theorem \ref{thm:main} does not require $f'(1)=0$, we could weaken the hypothesis $(f_1)$ into $0\le f'(1)< \lambda_2^{\mathrm{rad}}$. The only reason why we stated Theorem \ref{thm:main} under the stronger assumption $f'(1)=0$ is that, since $\lambda_2^{\mathrm{rad}}\searrow 0$ as $R\to\infty$, the hypothesis $R>R_k^*$ and $f'(1)< \lambda_2^{\mathrm{rad}}$ are in competition with each other, unless $f'(0)=0$, cf. also \cite[Remark 4.3]{ABF2}.
Secondly, the most evident difference between the two theorems is that, while in \cite{ABF3} we find $2k$ non-constant solutions sharing, in pairs, the same oscillatory behavior around the constant solution $u\equiv1$, in the present setting, we can find $4k$ non-constant solutions sharing, in groups of four, the same oscillatory behavior.  
A similar pattern of multiple solutions was found in \cite{ABF2} for a $p$-Laplacian Neumann problem with $1<p<2$, and, in the semilinear setting, for a Neumann Laplacian problem with a nonlinearity satisfying $(f_1)$. 

To explain where this difference originates, we need to briefly describe the technique used to prove the two theorems. As already mentioned, in both cases we use the shooting method for the equivalent ODE problem
\begin{equation}\label{eq:main_radial}
\begin{cases}
\left(r^{N-1} \frac{u'}{\sqrt{1-(u')^2}}\right)'+r^{N-1} f(u)=0 \qquad r\in (0,R)\\
u > 0 \\
u'(0)=u'(R)=0.
\end{cases}
\end{equation}
Namely, we rewrite the second-order equation in \eqref{eq:main_radial} as the equivalent first-order planar system
\begin{equation}\label{sys:intro}
\begin{cases}
&u' = \displaystyle{\frac{v}{r^{N-1}\sqrt{1+ (v/r^{N-1})^2}}},\smallskip\\
&v' = - r^{N-1} f(u),
\end{cases}
\end{equation}
coupled with the initial condition $(u(0),v(0)) = (d,0)$, 
and we look for values $d \in (0,+\infty) \setminus \{1\}$ such that the solution
$(u_d,v_d)$ satisfies $v_d(R) = 0$ (that is $u_d'(R)=0$, being $u_d$ ultimately a solution of \eqref{eq:main_radial}). Now, thanks to $(f_\textrm{eq})$, the solutions $(u_d,v_d)$, with $d\neq1$, of \eqref{sys:intro} turn clockwise around the equilibrium point $(1,0)$ in the phase plane, see Fig. \ref{fig:phase-plane}. Furthermore, the number of half-turns around such $(1,0)$ is exactly the number of zeros of $u_d-1$.
\begin{figure}[h]
\includegraphics[scale=1]{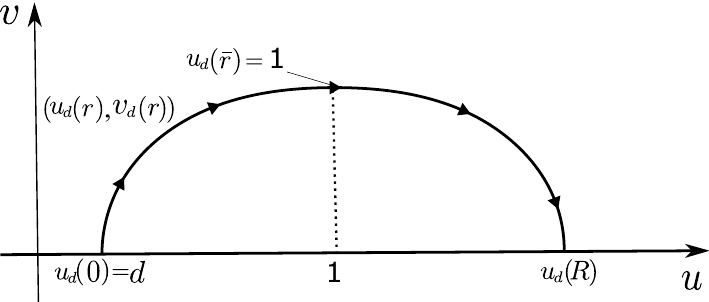}
\caption{The solution $(u_d,v_d)$ of \eqref{sys:intro} turns clockwise around $(1,0)$ in the phase plane.}\label{fig:phase-plane}
\end{figure}
Using condition $(f_1)$ or $(f_1)'$, it is possible to estimate the number of half-turns when the solution is shot from $d$ close to 1 (cf. Lemma \ref{lem:d_close1} below and \cite[Lemma 3.1]{ABF3}): \\
\centerline{$f'(1)<(\mbox{resp. }>)\,\lambda_{k+1}^{\mathrm{rad}}$\quad $\Longrightarrow$\quad $(u_d,v_d)$ performs less (resp. more) than $k$ half-turns.}
On the other hand, for $d=0$ the solution is constant ($u_0\equiv 0$) and so it performs zero half-turns around $(1,0)$. Finally, for $d$ large enough ($d\ge R+1$) the solution performs less than one half-turn, cf. Lemma \ref{le:d-large}.
Therefore, when $(f_1)'$ holds, we immediately have the multiplicity result and the precise oscillatory behavior using a continuity argument. Conversely, when $(f_1)$ holds, the situation is more involved, because the continuity argument, in general, does not ensure the existence of any non-constant solution. In this case, we adapt to the Neumann problem a technique used in \cite{BoGa-ccm} for a similar Dirichlet problem, to prove the existence, for every $k\in\mathbb N$ and for sufficiently large domains, of two initial data $(d^+_k)^*\in(1,R+1)$ and $(d^-_k)^*\in(0,1)$, such that the solutions of \eqref{sys:intro} shot from $(d^\pm_k)^*$ perform more than $k$ half-turns around $(1,0)$. This allows to use the continuity argument both on the left and on the right of each $(d^\pm_k)^*$, thus proving the existence of a double number of solutions with respect to the ones found under assumption $(f_1)'$, cf. Figg. \ref{fig:number-ht-1} and \ref{fig:number-ht-2}. 
\begin{figure}[h]
\includegraphics[scale=1]{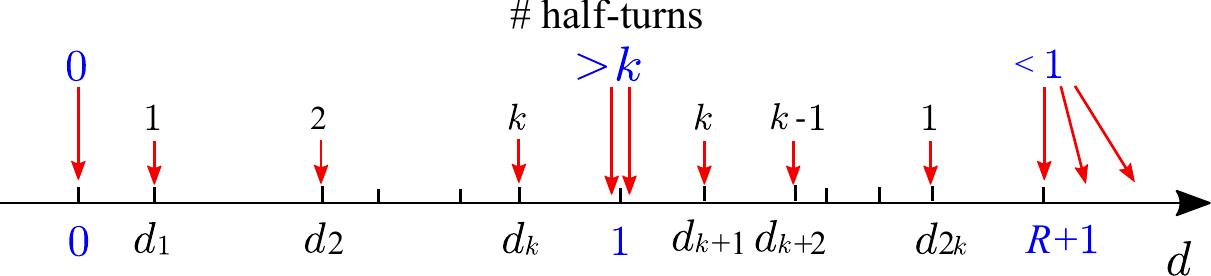}
\caption{Under the assumption $(f_1)'$: the number of half-turns around the point $(1,0)$ performed by the solution $(u_d,v_d)$ of \eqref{sys:intro}, when shot from different values of $d>0$.}
\label{fig:number-ht-1}
\end{figure}
\begin{figure}[h]
\includegraphics[scale=1]{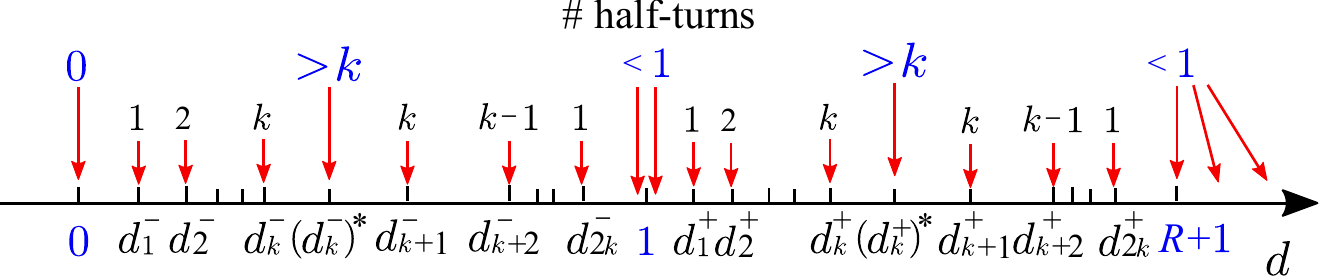}
\caption{Under the assumption $(f_1)$ and in domains $B_R$ sufficiently big: the number of half-turns around the point $(1,0)$ performed by the solution $(u_d,v_d)$ of \eqref{sys:intro}, when shot from different values of $d>0$.}
\label{fig:number-ht-2}
\end{figure}

We remark that all the results proved in this paper hold also in annular domains, where some proofs are simplified by the fact that the weight $r^{N-1}$ appearing in \eqref{sys:intro} is away from zero, cf. the proof of Theorem \ref{thm:main} and also \cite{ABF3}.

The plan of the paper is the following. In Section \ref{sec:2}, we prove that, if a solution $(u_d,v_d)$ of \eqref{sys:intro} is shot from some $d$ in a neighborhood of 1 or from some very large $d$, it performs less than one half-turn around $(1,0)$ in the phase plane. This result and its preliminary lemmas are essentially contained in \cite{ABF2,ABF3}. In Section \ref{sec:3}, we prove Theorem \ref{thm:main}, namely that, if the domain is sufficiently large, we can get as many oscillatory solutions as we want, and those solutions exhibit the same oscillatory behavior in groups of four. The results of Section \ref{sec:3} are inspired from the ones of \cite{BoGa-ccm}.

\section{The ``slow'' solutions of \eqref{sys:intro}}\label{sec:2}
For $f$ satisfying $(f_\textrm{reg})$, $(f_\textrm{eq})$ and $(f_1)$, let $\hat f$ denote its trivial continuous extension
\begin{equation}\label{eq:f_hat}
\hat f(s):=\begin{cases}
	f(s) \quad&\text{if } s \geq 0\\
	0 &\text{if } s < 0.
\end{cases}
\end{equation}
Let $\varphi(s):=\displaystyle{\frac{s}{\sqrt{1-s^2}}}$. We observe that $\varphi$ is invertible with inverse $\varphi^{-1}(t)=\displaystyle{\frac{t}{\sqrt{1+t^2}}}$,
and that
\begin{equation}\label{eq:inverse-phi}
|\varphi^{-1}(t)|<1 \qquad\text{for all } t\in\R.
\end{equation}
Since we are dealing with radial solutions, it is useful to consider the radial version of problem \eqref{eq:main}, with $f$ replaced by $\hat{f}$:
\begin{equation}\label{eq:radial}
\begin{cases}
(r^{N-1} \varphi(u'))'+r^{N-1}\hat f(u)=0 \qquad r\in (0,R)\\
u'(0)=u'(R)=0,
\end{cases}
\end{equation}
where the prime symbol $'$ denotes the derivative with respect to $r$. In view of the following maximum principle-type result, 
$u$ is a non-constant solution of \eqref{eq:main} if and only if $u$ is a non-constant solution of \eqref{eq:radial}.

\begin{lemma}[Lemma 2.3 of \cite{ABF3} and Lemma 2.1 of \cite{ABF}]\label{lem:hat_f} The function
$u$ is a radial solution of \eqref{eq:main} if and only if $u$ solves \eqref{eq:radial} and $u\not\equiv -C$ with $C\geq0$.
\end{lemma}

As described in the Introduction, we pass to the equivalent first-order planar system and we consider the associated Cauchy problem 
\begin{equation}\label{eq:shooting}
\begin{cases}
u'=\varphi^{-1}\left(\frac{v}{r^{N-1}}\right) \qquad & r\in (0,R)\\
v'=-r^{N-1} \hat f(u) \qquad & r\in (0,R) \\
u(0)=d, \quad v(0)=0
\end{cases}\qquad\qquad (d\ge 0).
\end{equation}

The following uniqueness, global continuability, continuous dependence from the initial data, and regularity result holds for \eqref{eq:shooting}.

\begin{lemma}[Lemma 2.1 of \cite{ABF3}]\label{lem:uniqueness_Cauchy}
For every $d\geq 0$, the local $W^{1,\infty}$ solution $(u_d,v_d)$ of \eqref{eq:shooting} is unique and can be defined on the whole $[0,R]$; moreover, $u_d$ is of class $C^2([0,R])$, with $u_d'(0) = 0$.

In addition, if $(d_n)\subset [0,+\infty)$ is such that $d_n\to d\in [0,+\infty)$ as $n\to+\infty$, then 
\begin{equation}\label{eq:dip_cont}
(u_{d_n}(r),v_{d_n}(r))\to (u_d(r),v_d(r))\quad \mbox{uniformly for }r\in [0,R],
\end{equation}
\begin{equation*}
u'_{d_n}(r)\to u'_d(r) \quad \mbox{uniformly for }r\in [0,R].
\end{equation*}
\end{lemma}

Thanks to the uniqueness stated in the previous lemma, we can pass to (clockwise) polar coordinates centered at $(1,0)$ for system \eqref{eq:shooting}:
\begin{equation}\label{eq:polar}
\begin{cases}
u(r)-1=\rho(r)\cos\theta(r) \\
v(r)=-\rho(r)\sin\theta(r)
\end{cases}\qquad \mbox{for }r\in [0,R].
\end{equation}
For $d\in [0,\infty)\setminus\{1\}$, if $(u_d,v_d)$ solves \eqref{eq:shooting}, the corresponding $(\theta_d,\rho_d)$ is such that $\theta_d$ satisfies the following differential equation in $(0,R)$,
\begin{equation}\label{eq:theta'}
\begin{aligned}
\theta_d'&=\frac{1}{\rho_d^2} \left[\varphi^{-1}\left(\frac{v_d}{r^{N-1}}\right)v_d + r^{N-1}\hat f(u_d)(u_d-1) \right] \\
&=\frac{\sin^2\theta_d}{r^{N-1}[1+(v_d/r^{N-1})^2]^{1/2}}
+r^{N-1} \hat f(u_d) \frac{u_d-1}{\rho_d^2}
\end{aligned}
\end{equation}
with initial conditions
\begin{equation}\label{eq:initial-cond}
\theta_d(0)=\begin{cases}
\pi \quad \text{if } 0<d<1\\
0 \quad \text{if } d>1.
\end{cases}
\quad\text{and}\quad
\rho_d(0)=|d-1|.
\end{equation}
By \eqref{eq:theta'} and $(f_{\mathrm{eq}})$, $\theta'_d(r)>0$ for every $r\in [0,R]$, so that the solution $(u_d,v_d)$ is actually turning clockwise around $(1,0)$ in the phase plane $(u,v)$; furthermore, by \eqref{eq:polar}, $u_d(r)=1$ if and only if $\theta_d(r)=\frac{\pi}{2}+ k\pi$ for some $k\in \mathbb Z$. Therefore, since the solutions $(u_d,v_d)$ of \eqref{eq:shooting} have $v_d(0)=0$, the number of half-turns of the solutions around $(1,0)$, is equal to the number of zeros of $u_d(r)-1$ in $(0,R)$, as anticipated in the Introduction (cf. Fig. \ref{fig:polar-coord}).
\begin{figure}[h]
\includegraphics[scale=1]{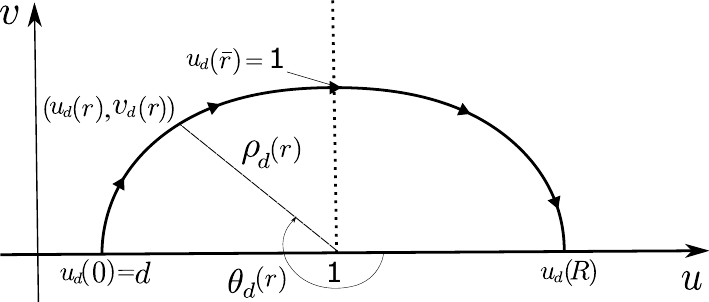}
\caption{The solution $(u_d,v_d)$ of \eqref{eq:shooting} and the polar coordinates $(\rho_d,\theta_d)$ in the phase plane introduced in \eqref{eq:polar}.
}\label{fig:polar-coord}
\end{figure}
We further remark that the continuous dependence, stated in Lemma \ref{lem:uniqueness_Cauchy} for $(u_d,v_d)$ from the initial data, continues to hold also for $(\rho_d,\theta_d)$, when passing to the description of the problem in polar coordinates.
  
Our next goal is to count the number of half-turns performed by a solution of \eqref{eq:shooting}, shot from $d$ in a neighborhood of 1. To this aim, we will estimate the quantity $\theta_d(R) - \theta_d(0)$ for $d$ close enough to 1. As it will be clear from the proof of Lemma \ref{lem:d_close1} below, hypothesis $(f_1)$ plays a crucial role in this estimate, together with a comparison with the
linear eigenvalue problem for the Laplacian in the ball $B_R$ under Neumann boundary conditions, namely
\begin{equation}\label{eq:eigenv_radial}
-(r^{N-1}u')'=\lambda r^{N-1} u\quad\mbox{in }(0,R), \qquad u'(0)=u'(R)=0.
\end{equation}
We consider the change of variables 
\[
\begin{cases}
u(r)=\varrho_\lambda(r)\cos\vartheta_\lambda(r) \\
r^{N-1}u'(r)=-\varrho_\lambda(r)\sin\vartheta_\lambda(r).
\end{cases}
\]
If $u_\lambda$ solves \eqref{eq:eigenv_radial}, its polar coordinates $(\vartheta_\lambda,\varrho_\lambda)$ are such that
\begin{equation}\label{eq:theta'-omogenea-associata}
\vartheta_\lambda'=\frac{\sin^2\vartheta_\lambda}{r^{N-1}}+\lambda r^{N-1}\cos^2\vartheta_\lambda>0, \qquad r\in [0,R].
\end{equation}
Therefore, the angular variable $\vartheta_\lambda(r)$ is strictly increasing in $r$. Moreover, by convention, we consider eigenfunctions $u_\lambda$ with $u_\lambda(0)>0$, thus $\vartheta_\lambda(0)=0$. 

We recall a monotonicity result from \cite{RW99} (see Theorem 4 therein):
\begin{equation}\label{eq:monotonicity}
\mbox{$\vartheta_\lambda(R)$ is strictly increasing in $\lambda$.}
\end{equation} 
Moreover, the eigenfunctions of \eqref{eq:eigenv_radial} satisfy the classical Sturm theory, as stated in the following Theorem.

\begin{theorem}{\cite[Theorem 1]{RW99}}\label{th:eigen}
The problem \eqref{eq:eigenv_radial} has a countable number of simple eigenvalues $0=\lambda_1^{\mathrm{rad}}<\lambda_2^{\mathrm{rad}}<\lambda_3^{\mathrm{rad}}<\dots$, $\lim_{k\to+\infty}\lambda_k^{\mathrm{rad}}=+\infty$, and no other eigenvalues. The eigenfunction $u_{k}$ that corresponds to the $k$-th eigenvalue $\lambda_k^{\mathrm{rad}}$ has exactly $k-1$ simple zeros in $(0,R)$. Namely, 
its angular variable satisfies 
\begin{equation}\label{initialcondition-ep}
\vartheta_{\lambda^{\mathrm{rad}}_{k}}(R) = (k-1)\pi
\quad \mbox{for every integer } k\geq1.
\end{equation}
\end{theorem}

We are now ready to count the number of half-turns performed by a solution $(u_d,v_d)$ of \eqref{eq:shooting}, with $d$ close to 1.

We prove the following lemma under a weaker hypothesis on $f$ than $(f_1)$, because, as mentioned in the Introduction, the arguments in the proofs of Theorem \ref{thm:main} and of the previous lemmas continue to hold even under the weaker assumption $f'(1)<\lambda_2^{\mathrm{rad}}$.

\begin{lemma}[cf. Lemma 2.5 of \cite{ABF2}]\label{lem:d_close1}
Suppose that, for some integer $k\geq1$, $f'(1)<\lambda_{k+1}^{\textnormal{rad}}$.
Then, there exists $\bar\delta>0$ such that $\theta_d(R) - \theta_d(0)<k\pi$ for $d\in [1-\bar\delta,1+\bar\delta]\setminus\{1\}$.

In particular, if $(f_1)$ holds, $\theta_d(R) - \theta_d(0)<\pi$ for $d\in [1-\bar\delta,1+\bar\delta]\setminus\{1\}$, that is $(u_d,v_d)$ performs less than one half-turn around $(1,0)$.
\end{lemma}
\begin{proof}
Let $\bar\lambda, \varepsilon>0$ be such that
\begin{equation}\label{eq:bar_lambda_def}
f'(1)+\varepsilon \leq \bar \lambda < \lambda_{k+1}^{\mathrm{rad}}.
\end{equation}
Then, using assumptions $(f_\textrm{reg})$ and $(f_\textrm{eq})$, there exists $\delta>0$ such that, for every $s$ satisfying $|s-1|\leq \delta$, it holds
\begin{equation}\label{eq:ineq1}
\hat f(s)(s-1)\leq (f'(1)+\varepsilon)(s-1)^2 \leq
\bar\lambda (s-1)^2.
\end{equation}
Thanks to \eqref{eq:dip_cont}, there exists $\bar\delta>0$ such that, for every $d\neq1$ satisfying $|d-1|\leq\bar\delta$, it holds
\begin{equation}\label{eq:ineq2}
|u_d(r)-1|\leq \delta \quad\mbox{for every }r\in [0,R],
\end{equation}
being $u_{1}\equiv 1$ in $[0,R]$. 
By replacing \eqref{eq:ineq1} and \eqref{eq:ineq2} into \eqref{eq:theta'}, and recalling \eqref{eq:polar}, we obtain that, for every $d$ satisfying $0<|d-1|\leq \bar\delta$ and $r\in [0,R]$,
\begin{equation}
\theta_d'(r)\leq \frac{\sin^2\theta_d(r)}{r^{N-1}}+\bar\lambda r^{N-1} \cos^2\theta_d(r).
\end{equation}
Using equation \eqref{eq:theta'-omogenea-associata} with $\lambda=\bar\lambda$, the Comparison Theorem for ODEs  (see [Lemma 4, RW99]), and recalling that $\vartheta_{\bar\lambda}(0)=0$, we obtain, for all $d\neq 1$ satisfying $0<|d-1|\leq\bar\delta$,
$$
\theta_d(r) - \theta_d(0)\leq \vartheta_{\bar\lambda}(r)\quad\mbox{for all }r\in[0,R].
$$
In particular, by the fact that $\theta'_d>0$, relation \eqref{eq:bar_lambda_def}, the monotonicity \eqref{eq:monotonicity} and Theorem \ref{th:eigen}, we have
\[
\theta_d(R) - \theta_d(0)<  \vartheta_{\lambda_{k+1}^{\mathrm{rad}}}(R) = k\pi.  \qedhere
\]
\end{proof}

Up to now, we have found that solutions of \eqref{eq:shooting} shot from some $d$ close to 1 are very ``slow'', in the sense that, in the interval $[0,R]$, they cannot even complete one half-turn. We also know that, if the solution is shot exactly from $d=0$, it is even slower (it is constant!) and it performs exactly zero half-turns
around the point $(1,0)$. In the next lemma, arguing as in the proof of \cite[Theorem 1.1]{ABF3}, we will prove that also solutions shot from $d$ large are very ``slow''. Here the singular character of the Minkowski-curvature operator and, in particular, relation \eqref{eq:inverse-phi} play an important role. 

\begin{lemma}\label{le:d-large}
If $d\ge 1+R$, the solution $(u_d,v_d)$ of \eqref{eq:shooting} performs less than one half-turn around $(1,0)$ in the phase plane. Equivalently, $\theta_d(R)<\pi$ if $d\ge 1+R$.
\end{lemma}
\begin{proof}
By \eqref{eq:inverse-phi}, we get for every $r\in[0,R]$
$$u_d(r)\ge d -\int_0^r|u'_d(s)|ds\ge d-R \ge 1.$$
Since $(u_d,v_d)$ is turning around $(1,0)$ and $u_d(0)>1$, this proves that $v_d(r)\neq0$ for every $r\in [0,R]$. 
\end{proof}

\section{The ``fast'' solutions of \eqref{sys:intro} and the proof of Theorem \ref{thm:main}}\label{sec:3}
In this section we prove a sufficient condition on the size of the domain $B_R$, to get multiplicity of oscillatory solutions to \eqref{eq:main}. In the previous section, we found that the Cauchy problem \eqref{eq:shooting} does not have any solutions $(u_d,v_d)$ with $v_d(R)=0$ if $d\in (0,\delta]\cup[1-\bar\delta,1)\cup(1,1+\bar\delta]\cup[1+R,+\infty)$, for $\delta>0$ sufficiently small and $\bar \delta$ as in Lemma \ref{lem:d_close1}. 
Adapting a method introduced in \cite{BosZan13}, we are able to prove that, if the radius of the ball $B_R$ is sufficiently large, there exist two initial values, $(d_k^-)^*\in(\delta,1-\bar\delta)$ and $(d_k^+)^*\in(1+\bar\delta,1+R)$, such that the solutions $(u_{(d_k^\pm)^*},v_{(d_k^\pm)^*})$ of \eqref{eq:shooting} turn around $(1,0)$ more than $k$ half-times. The estimate is performed using two spiral-like curves which bound the solution either from below or from above in each quarter of the phase plane. Once we have proved the existence of such $(d_k^\pm)^*$, Theorem \ref{thm:main} immediately follows by a continuation argument, cf. Fig. \ref{fig:number-ht-2}.  

We report below the statement of the general result that uses the method of the spiral-like curves, in the version proved in \cite{BoGa-ccm}.
\begin{lemma}[Proposition 2.1 of \cite{BoGa-ccm}]\label{lem:spirals}
Let $a_i,\,b_i:(-\delta,\delta)\to\mathbb R$, $i\in\{1,2\}$, $\delta>0$, be two locally Lipschitz functions verifying
\begin{equation}\label{eq:ai-bi}
0 < a_1(s)s \le b_1(s)s \quad\mbox{and}\quad 0 < b_2(s)s \le a_2(s)s \quad\mbox{for every } s\in(-\delta,\delta)\setminus\{0\}.
\end{equation}
Then, for every $k\in\mathbb N$, $k\ge 1$, there exist $\tau_k^*>0$ and $\rho_k^*\in(0,\delta)$ such that for every interval $I=[r_0,r_1]$, with $r_1-r_0>\tau_k^*$, and for every couple of locally Lipschitz functions $X,\,Y :I\times\mathbb R\to\mathbb R$ satisfying
\begin{equation}\label{eq:hpX}
a_1(y)y \le X(r,y)y \le b_1(y)y \quad\mbox{for every } (r,y)\in I\times(-\delta,\delta),
\end{equation}
\begin{equation}\label{eq:hpY}
b_2(x)x \le Y(r,x)x \le a_2(x)x \quad\mbox{for every } (r,x)\in I\times(-\delta,\delta),
\end{equation}
it holds that every solution $(x(r),y(r))$ defined in $I$ of 
$$
\begin{cases}
x'&=X(r,y),\\
y'&=-Y(r,x),\\
\end{cases}
$$
with 
$x(r_0)^2+y(r_0)^2=(\rho_k^*)^2$,
satisfies
\begin{itemize}
\item[(i)] $x(r)^2+y(r)^2>0$ for every $r\in I$;
\item[(ii)] $\theta(r_1)-\theta(r_0)>k\pi$, 
\end{itemize}
 where $(\rho,\theta)$ are the polar coordinates of $(x,y)$ centered at $(0,0)$, namely $x(r)=\rho(r)\cos(\theta(r))$ and $y(r)=-\rho(r)\sin(\theta(r))$.
\end{lemma}

As in \cite{BoGa-ccm}, we introduce the following auxiliary Cauchy problem: 
\begin{equation}\label{eq:shooting-aux}
\begin{cases}
u'=\tilde\varphi^{-1}\left(\frac{v}{r^{N-1}}\right) \qquad & r\in (0,R)\\
v'=-r^{N-1} \tilde f(u) \qquad & r\in (0,R) \\
u(0)=d, \quad v(0)=0
\end{cases}\qquad\qquad (d\ge 0),
\end{equation}
where $\tilde f:\mathbb R \to\mathbb R$ is a locally Lipschitz function such that 
$$
\tilde f(s):=
\begin{cases}
\hat f(s)\quad&\mbox{if }|s|\le 1+R,\\
0&\mbox{if }|s|\ge 2+R.
\end{cases}
$$
Let $M:=\max_{s\in\mathbb R}|\tilde f(s)|$ and $\gamma:=\varphi^{-1}(MR)\in(0,1)$. The $C^1$-function $\tilde \varphi:\mathbb R\to\mathbb R$ is defined as follows: 
$$
\tilde\varphi(s):=
\begin{cases}
\varphi(s)\quad&\mbox{if }|s|\le \gamma,\\
\varphi'(\gamma)(s-\gamma)-\varphi(\gamma)&\mbox{if }s<-\gamma,\\
\varphi'(\gamma)(s-\gamma)+\varphi(\gamma)&\mbox{if }s>\gamma.
\end{cases}
$$
We observe that $\tilde\varphi$ is odd and strictly increasing and so, also its inverse enjoys the same properties. 

As for \eqref{eq:shooting}, also for \eqref{eq:shooting-aux} it is possible to prove global existence, uniqueness, and continuous dependence on the initial data of the solution, cf. \cite[Lemma 3.2 with $\lambda=1$]{BoGa-ccm}. Furthermore, we prove below that the oscillatory solutions of \eqref{eq:shooting-aux} solve also \eqref{eq:shooting}.  
\begin{lemma}\label{le:Paux-to-P}
Let $(u,v)$ be a solution of \eqref{eq:shooting-aux} such that $u\in C^1([0,R])$ and $u(\bar r)=1$ for some $\bar r\in (0,R)$. Then, $(u,v)$ solves \eqref{eq:shooting}.
\end{lemma}
\begin{proof}
Following the argument in the proof of \cite[Lemma 3.1]{BoGa-ccm}, for every $r\in[0,R]$, we integrate the equation for $v$ in \eqref{eq:shooting-aux} to get
$$
r^{N-1}\tilde\varphi(u'(r))= -\int_0^r s^{N-1}\tilde f(u(s))ds. 
$$
Thus, using the properties of $\tilde\varphi^{-1}$,
\begin{equation}\label{eq:tildevarphi}
\begin{aligned}
|u'(r)|&= \left|-\tilde\varphi^{-1}\left(\frac{1}{r^{N-1}}\int_0^r s^{N-1}\tilde f(u(s))ds\right)\right|\\
&\le \tilde\varphi^{-1}\left(\int_0^r \left|\left(\frac{s}{r}\right)^{N-1}\tilde f(u(s))\right|ds\right)\le \tilde\varphi^{-1}\left(\int_0^R M ds\right)=\gamma.
\end{aligned}
\end{equation}
Therefore, $\tilde\varphi(u'(r))=\varphi(u'(r))$, that is $u'(r)=\varphi^{-1}\left(\frac{v}{r^{N-1}}\right)$ for every $r\in [0,R]$. On the other hand, since $u(\bar r)=1$, by \eqref{eq:tildevarphi}, we get for every $r\in[0,R]$
$$
|u(r)|= \left|1+\int_{\bar r}^r u'(s) ds\right|\le 1+\int_0^R|u'(s)|ds \le \gamma R+1<R+1,
$$
and so $\tilde f(u(r))=\hat f(u(r))$, that is $v'(r)=-r^{N-1}\hat f(u(r))$ for every $r\in[0,R]$. 
\end{proof}

Thanks to the uniqueness, also in \eqref{eq:shooting-aux} we can pass to polar coordinates $(\rho,\theta)$ centered at $(1,0)$ as in \eqref{eq:polar}. If $(u_d,v_d)$ for some $d\in [0,\infty)\setminus\{1\}$ solves \eqref{eq:shooting-aux}, its angular variable $\theta_d$ solves the following differential equation
\begin{equation}\label{eq:tilde-theta'}
\begin{aligned}
\theta_d'&=\frac{1}{\rho_d^2} \left[\tilde\varphi^{-1}\left(\frac{v_d}{r^{N-1}}\right)v_d + r^{N-1}\tilde f(u_d)(u_d-1) \right],
\end{aligned}
\end{equation}
with initial conditions as in \eqref{eq:initial-cond}. Again, $\theta_d'(r)\ge 0$ by $(f_{\mathrm{eq}})$ and the definitions of $\tilde f$ and $\tilde \varphi$, so that the solution $(u_d,v_d)$ turns clockwise around $(1,0)$ in the phase plane $(u,v)$. 

We are now ready to prove the main result of this paper.
\begin{proof}[$\bullet$ Proof of Theorem \ref{thm:main}]
We want to apply Lemma \ref{lem:spirals} with
$$
x(r)=u(r)-1,\quad y(r)=v(r),\quad X(r,y)=\tilde\varphi^{-1}\left(\frac{y}{r^{N-1}}\right),\quad Y(r,x)=r^{N-1}\tilde f(x+1).
$$
In order to let conditions \eqref{eq:ai-bi}, \eqref{eq:hpX} and \eqref{eq:hpY} be satisfied, we need the factor $r^{N-1}$ to be away from zero. Thus, let $r_0$ be any constant such that $0<r_0<R$, and consider the interval $I=[r_0,R]$. 
Let $0<\delta < \min\{1,R\}$. 
If we define, for $i\in \{1,2\}$, the locally Lipschitz functions $a_i,\,b_i:(-\delta,\delta)\to \mathbb R$ as follows  
$$a_1(s):=\frac{s}{\varphi'(\gamma)R^{N-1}}, \quad b_1(s):=\frac{s}{r_0^{N-1}}, \quad a_2(s):=R^{N-1}\tilde f(s+1), \quad b_2(s):=r_0^{N-1}\tilde f(s+1),$$
condition \eqref{eq:ai-bi} is clearly satisfied (notice that $\varphi'(\gamma)=(1+M^2R^2)^{3/2}>1$). Furthermore, for every $r\in I$ and $(u-1,v)\in(-\delta,\delta)\times(-\delta,\delta)$ the following conditions hold:
$$
\begin{aligned}
&\frac{v^2}{\varphi'(\gamma)R^{N-1}}\le \tilde\varphi^{-1}\left(\frac{v}{r^{N-1}}\right)v \le \frac{v^2}{r_0^{N-1}},\\
&r_0^{N-1}\tilde f(u)(u-1) \le  r^{N-1}\tilde f(u)(u-1)\le R^{N-1}\tilde f(u)(u-1),
\end{aligned}
$$
where we used that for every $s\in\mathbb R$, $f(s)(s-1)\ge0$ by $(f_{\mathrm{eq}})$, $\tilde\varphi(s)s\ge0$, and 
$$
\frac{s^2}{\varphi'(\gamma)} \le \tilde\varphi^{-1}(s)s \le s^2\quad\mbox{for every }s\in \mathbb R.
$$
Therefore, since all hypotheses of Lemma \ref{lem:spirals} are satisfied, for every integer $k\ge 1$, there exist $R^*_k:=\tau_k^*$ and $\rho_k^*\in(0,\delta)$ such that for every  solution $(u(r),v(r))$ of 
\begin{equation}\label{eq:Paux}
\begin{cases}
u'&=\tilde\varphi^{-1}\left(\frac{v}{r^{N-1}}\right),\\
v'&=-r^{N-1}\tilde f(u)
\end{cases}
\end{equation}
defined in $[r_0,R]$, such that $(u(r_0)-1)^2+v(r_0)^2=(\rho_k^*)^2$, the corresponding angular variable verifies
\begin{equation}\label{eq:thetar0}
\theta(R)-\theta(r_0)>k\pi.
\end{equation}
Now, since $\delta<R$, and using the fact that \eqref{eq:shooting-aux} admits the constant solutions $(u_0,v_0)\equiv(0,0)$, $(u_1,v_1)\equiv(1,0)$ and $(u_{R+2},v_{R+2})\equiv (R+2,0)$, by continuous dependence (\cite[Lemma 3.2]{BoGa-ccm}) there exist $(d_k^-)^*\in (0,1)$ and $(d_k^+)^*\in(1,R+1)$ such that the solution $(u(r),v(r))$, defined in $I$, actually comes from a solution $(u_{(d_k^\pm)^*},v_{(d_k^\pm)^*})$ of \eqref{eq:shooting-aux} defined in the whole interval $[0,R]$:
$$
(u_{(d_k^\pm)^*}(r_0)-1)^2+v_{(d_k^\pm)^*}(r_0^2)=(\rho_k^*)^2.
$$
Then, recalling that $\theta_{(d_k^\pm)^*}'\ge 0$ (cf. \eqref{eq:tilde-theta'}), we obtain
$$
\tilde\theta_{(d_k^\pm)^*}(R)-\tilde\theta_{(d_k^\pm)^*}(0)\ge \tilde\theta_{(d_k^\pm)^*}(R)-\tilde\theta_{(d_k^\pm)^*}(r_0)>k\pi.
$$
This means that the two functions $u_{(d_k^\pm)^*}-1$ have more than $k$ zeros, with $k\geq1$. So, by Lemma \ref{le:Paux-to-P}, we know that they actually solve \eqref{eq:shooting}.
Therefore, using Lemmas \ref{lem:d_close1} and \ref{le:d-large}, the fact that $\theta_0\equiv 0$ in $[0,R]$, and the continuous dependence \eqref{eq:dip_cont}, we get the existence of $4k$ initial data $d_j^\pm$ ordered as follows (cf. also Fig. \ref{fig:number-ht-2})
\begin{equation}\label{eq:dkpm}
\begin{aligned}
0&<d^-_1<d^-_2<\dots<d^-_k<(d^-_k)^*<d^-_{k+1}<d^-_{k+2}<\dots<d^-_{2k}\\
& <1<d^+_1<d^+_2<\dots<d^+_k<(d^+_k)^*<d^+_{k+1}<d^+_{k+2}<\dots<d^+_{2k}<R+1,
\end{aligned}
\end{equation}
such that every solution $(u^{\pm}_{j},v^{\pm}_{j}):=(u_{d_j^\pm},v_{d_j^{\pm}})$ of \eqref{eq:shooting} has $(u^\pm_j)'(R)=0$, and moreover 
\begin{equation}\label{eq:zeros}
u^\pm_j(r)-1 \quad\mbox{and}\quad u^\pm_{2k+1-j}(r)-1 \quad\mbox{have exactly $j$ zeros for every $j=1,\dots,k$.} 
\end{equation}
Clearly, being oscillating, those solutions are non-constant. In conclusion, by Lemma \ref{lem:hat_f}, $u^\pm_j$ are solutions of \eqref{eq:main} satisfying the conditions (i), (ii) and (iii) of Theorem \ref{thm:main}. 
\end{proof}

\section*{Acknowledgments}
\noindent This research was partially supported by the INdAM - GNAMPA Project 2019 ``Il modello di Born-Infeld per l'elettromagnetismo nonlineare: esistenza, regolarit\`a e molteplicit\`a di soluzioni". 

\bibliographystyle{alpha}
\bibliography{biblio}

\def\cprime{$'$}
\begin{thebibliography}{COOR13}

\bibitem[Azz14]{Az-14}
Antonio Azzollini.
\newblock Ground state solution for a problem with mean curvature operator in
  {M}inkowski space.
\newblock {\em J. Funct. Anal.}, 266(4):2086--2095, 2014.

\bibitem[Azz16]{Az-16}
Antonio Azzollini.
\newblock On a prescribed mean curvature equation in {L}orentz-{M}inkowski
  space.
\newblock {\em J. Math. Pures Appl.}, 106(6):1122--1140, 2016.

\bibitem[BCF19]{BoCoFo-pp}
Denis Bonheure, Francesca Colasuonno, and Juraj F\"oldes.
\newblock On the {B}orn-{I}nfeld equation for electrostatic fields with a
  superposition of point charges.
\newblock {\em Ann. Mat. Pura Appl.}, 198(3):749--772, 2019.

\bibitem[BCN18]{ABF}
Alberto Boscaggin, Francesca Colasuonno, and Benedetta Noris.
\newblock Multiple positive solutions for a class of $p$-{L}aplacian {N}eumann
  problems without growth conditions.
\newblock {\em ESAIM Control Optim. Calc. Var.}, 24(4):1625--1644, 2018.

\bibitem[BCN19]{ABF2}
Alberto Boscaggin, Francesca Colasuonno, and Benedetta Noris.
\newblock A priori bounds and multiplicity of positive solutions for a
  $p$-{L}aplacian {N}eumann problem with sub-critical growth.
\newblock {\em Proc. Roy. Soc. Edinburgh Sect. A}, 2019.
\newblock doi: 10.1017/prm.2018.143.

\bibitem[BCN20]{ABF3}
Alberto Boscaggin, Francesca Colasuonno, and Benedetta Noris.
\newblock Positive radial solutions for the minkowski-curvature equation with
  neumann boundary conditions.
\newblock {\em Discrete Contin. Dyn. Syst. Ser. S}, 2020.
\newblock doi: 10.3934/dcdss.2020150.

\bibitem[BdP16]{BdAP-16}
Denis Bonheure, Pietro d'Avenia, and Alessio Pomponio.
\newblock On the electrostatic {B}orn-{I}nfeld equation with extended charges.
\newblock {\em Comm. Math. Phys.}, 346(3):877--906, 2016.

\bibitem[BF18]{BoFe-pp}
Alberto Boscaggin and Guglielmo Feltrin.
\newblock Positive periodic solutions to an indefinite {M}inkowski-curvature
  equation.
\newblock arXiv preprint arXiv:1805.06659, 2018.

\bibitem[BG19]{BoGa-ccm}
Alberto Boscaggin and Maurizio Garrione.
\newblock Pairs of nodal solutions for a {M}inkowski-curvature boundary value
  problem in a ball.
\newblock {\em Commun. Contemp. Math.}, 21(2):1850006, 18, 2019.

\bibitem[BI34]{BI}
Max Born and Leopold Infeld.
\newblock Foundations of the new field theory.
\newblock {\em Proc. R. Soc. Lond. Ser. A}, 144(852):425--451, 1934.

\bibitem[BI19]{BoIa-pp}
Denis Bonheure and Alessandro Iacopetti.
\newblock On the regularity of the minimizer of the electrostatic
  {B}orn-{I}nfeld energy.
\newblock {\em Arch. Ration. Mech. Anal.}, 232(2):697--725, 2019.

\bibitem[BJM09]{BeJeMa-09}
Cristian Bereanu, Petru Jebelean, and J.~Mawhin.
\newblock Radial solutions for some nonlinear problems involving mean curvature
  operators in {E}uclidean and {M}inkowski spaces.
\newblock {\em Proc. Amer. Math. Soc.}, 137(1):161--169, 2009.

\bibitem[BJT13a]{BeJeTo-13}
Cristian Bereanu, Petru Jebelean, and Pedro~J. Torres.
\newblock Multiple positive radial solutions for a {D}irichlet problem
  involving the mean curvature operator in {M}inkowski space.
\newblock {\em J. Funct. Anal.}, 265(4):644--659, 2013.

\bibitem[BJT13b]{BeJeTo-13a}
Cristian Bereanu, Petru Jebelean, and Pedro~Jean Torres.
\newblock Positive radial solutions for {D}irichlet problems with mean
  curvature operators in {M}inkowski space.
\newblock {\em J. Funct. Anal.}, 264(1):270--287, 2013.

\bibitem[BM07]{BeMa-07}
Cristian Bereanu and Jean Mawhin.
\newblock Existence and multiplicity results for some nonlinear problems with
  singular {$\phi$}-{L}aplacian.
\newblock {\em J. Differential Equations}, 243(2):536--557, 2007.

\bibitem[BS83]{BaSi-8283}
Robert Bartnik and Leon Simon.
\newblock Spacelike hypersurfaces with prescribed boundary values and mean
  curvature.
\newblock {\em Comm. Math. Phys.}, 87(1):131--152, 1982/83.

\bibitem[BZ13]{BosZan13}
Alberto Boscaggin and Fabio Zanolin.
\newblock Pairs of nodal solutions for a class of nonlinear problems with
  one-sided growth conditions.
\newblock {\em Adv. Nonlinear Stud.}, 13(1):13--53, 2013.

\bibitem[CCOO12]{CoCoObOm-12}
Isabel Coelho, Chiara Corsato, Franco Obersnel, and Pierpaolo Omari.
\newblock Positive solutions of the {D}irichlet problem for the one-dimensional
  {M}inkowski-curvature equation.
\newblock {\em Adv. Nonlinear Stud.}, 12(3):621--638, 2012.

\bibitem[CCR14]{CoelhoCorsatoRivetti}
Isabel Coelho, Chiara Corsato, and Sabrina Rivetti.
\newblock Positive radial solutions of the dirichlet problem for the
  minkowski-curvature equation in a ball.
\newblock {\em Topological Methods in Nonlinear Analysis}, 44(1):23--39, 2014.

\bibitem[COOR13]{CoObOmRi-13}
Chiara Corsato, Franco Obersnel, Pierpaolo Omari, and Sabrina Rivetti.
\newblock Positive solutions of the {D}irichlet problem for the prescribed mean
  curvature equation in {M}inkowski space.
\newblock {\em J. Math. Anal. Appl.}, 405(1):227--239, 2013.

\bibitem[DW17]{DW}
Guowei Dai and Jun Wang.
\newblock Nodal solutions to problem with mean curvature operator in
  {M}inkowski space.
\newblock {\em Differential Integral Equations}, 30(5-6):463--480, 2017.

\bibitem[Eck86]{Ecker}
Klaus Ecker.
\newblock Area maximizing hypersurfaces in {M}inkowski space having an isolated
  singularity.
\newblock {\em Manuscripta Math.}, 56(4):375--397, 1986.

\bibitem[Ger83]{Ge-83}
Claus Gerhardt.
\newblock {$H$}-surfaces in {L}orentzian manifolds.
\newblock {\em Comm. Math. Phys.}, 89(4):523--553, 1983.

\bibitem[Maw13]{Ma-13}
Jean Mawhin.
\newblock Resonance problems for some non-autonomous ordinary differential
  equations.
\newblock In {\em Stability and bifurcation theory for non-autonomous
  differential equations}, volume 2065 of {\em Lecture Notes in Math.}, pages
  103--184. Springer, Heidelberg, 2013.

\bibitem[RW99]{RW99}
Wolfgang Reichel and Wolfgang Walter.
\newblock Sturm-{L}iouville type problems for the {$p$}-{L}aplacian under
  asymptotic non-resonance conditions.
\newblock {\em J. Differential Equations}, 156(1):50--70, 1999.

\end{thebibliography}
\end{document}